\documentclass[a4paper]{amsart}
\usepackage[utf8]{inputenc}
\usepackage[T2A]{fontenc}
\usepackage[english]{babel}

\usepackage{amssymb}
\usepackage{amsthm}

\ifx\pdfoutput\undefined
\usepackage{graphicx}
\else
\usepackage[pdftex]{graphicx}
\fi

\newtheorem{lemma}{Lemma}
\newtheorem{theorem}{Theorem}
\newtheorem{definition}{Definition}
\newtheorem{proposition}{Proposition}
\newtheorem{remark}{Remark}

\def\R{{\mathbb R}}

\def\BA{{\mathcal B \mathcal A}}
\def\A{{\mathcal A}}

\def\M{{\mathcal M}}

\def\conv{\mathop{\mathrm{conv}}}

\def\eps{\varepsilon}

\def\const{\mathop{\mathrm{const}}}

\def\le{\leqslant}
\def\ge{\geqslant}

\usepackage[unicode]{hyperref}

\begin{document}

\title{Linear programming over exponent pairs}
\author{Andrew V. Lelechenko}
\address{I.~I.~Mechnikov Odessa National University}
\email{1@dxdy.ru}

\keywords{Exponent pairs, divisor problem, linear programming}
\subjclass[2010]{
90C05, 
11Y16, 
11L07, 
11N37} 

\begin{abstract}
We consider the problem of the computation of
$\inf_p \theta p$
over the set of exponent pairs $P \ni p$
under linear constraints
for a certain class of objective functions $\theta$.
An effective algorithm is presented. The output
of the algorithm leads to the improvement
and establishing new estimates in the various
divisor problems in the analytic number theory.
\end{abstract}

\maketitle

\section{Introduction}

Exponent pairs are an extremely important concept
in the analytic number theory. They are defined implicitly.

\begin{definition}[{\cite[Ch. 2]{kratzel1988}}]
A pair $(k,l)$ of real numbers is called an exponent pair
if~$0\le k \le 1/2\le l\le 1$, and if for each~$s>0$ there exist integer $r>4$ and real $c\in(0,1/2)$
depending only on $s$
such that the inequality
$$
\sum_{a<n\le b} e^{2\pi i f(n)} \ll z^k a^l
$$
holds with respect to $s$ and $u$ when
the following conditions are satisfied:
$$
u>0,
\qquad
1 \le a < b < a u,
\qquad
y>0,
\qquad
z = y a^{-s} > 1;
$$
$f(t)$ being any real function with differential coefficients
of the first $r$ orders in~$[a,b]$ and
$$
\left| f^{(\nu+1)}(t) - y {d^\nu \over dt^\nu } t^{-s} \right|
< (-1)^\nu c y {d^\nu \over dt^\nu} t^{-s}
$$
for $a\le t\le b$ and $0 \le \nu \le r-1$.
\end{definition}

But for the computational purposes more explicit construction
is needed.

\begin{proposition}\label{pr:exp-pairs}
The set of the exponent pairs includes a convex hull $\conv P$
of the set~$P$ such that
\begin{enumerate}

\item $P$ includes a subset of initial elements $P_0$, namely
	\begin{enumerate}
	\item $(0,1)$ \cite{kratzel1988},
	\item $(2/13+\eps,35/52+\eps)$,
				$(13/80+\eps,1/2+13/80+\eps)$,
				$(11/68+\eps,1/2+11/68+\eps)$ \cite{huxley1989},
	\item $(9/56+\eps,1/2+9/56+\eps)$ \cite{huxley1988},
	\item $(89/560+\eps,1/2+89/560+\eps)$ \cite{watt1989},
	\item $H_{05}:=(32/205+\eps,1/2+32/205+\eps)$ \cite{huxley2005}.
	\end{enumerate}

\item
	$A(k,l)\in P$ and $BA(k,l) \in P$ for every $(k,l)\in P$,
	where operators $A$ and~$B$ are defined as follows:
	$$
	A(k,l) = \left( {k\over2(k+1)}, {k+l+1\over2(k+1)} \right),
	\qquad
	B(k,l) = \left( l-1/2, k+1/2 \right).
	$$

\end{enumerate}
\end{proposition}

Possibly the set of the exponent pairs includes elements
$(k,l) \not\in \conv P$, but at least~$\conv P$ incorporates all
{\em currently known} exponent pairs. Everywhere below writing
``a set of exponent pairs'' we mean $P$ in fact.

Denote by $Pp$ a set of exponent pairs, generated from the pair $p$
with the use of operators $A$ and $BA$. One can check that {\em currently}
$$
\conv P = \conv\bigr (PH_{05} \cup \{(0,1),(1/2,1/2)\} \bigr).
$$

\medskip

Many asymptotic questions of the number theory (especially in the
area of divisor problems) come to the optimization task
\begin{equation}\label{eq:opt-problem}
\inf_{(k,l) \in \conv P} \bigl\{
\theta(k,l) \bigm|
R_i(\alpha_i k + \beta_i l + \gamma_i),
\quad i=1,\ldots,j
\bigr\},
\end{equation}
where $\alpha_i, \beta_i, \gamma_i \in \mathbb{R}$,
$R_i\in {R_{>}, R_{\ge}}$, the predicate $R_{>}$ checks whether its argument is a positive value and $R_{\ge}$ checks whether its argument is non-negative, $i=1,\ldots, j$.

Graham \cite{graham1986} gave an effective method,
which in many cases is able to determine
\begin{equation*}\label{eq:opt-problem-(0,1)}
\inf_{(k,l)\in \conv P(0,1)} \theta(k,l)
\end{equation*}
with a given precision (and even exactly for certain values of $\theta$),
where
$$
\theta \in \Theta := \left\{
(k,l) \mapsto {ak+bl+c \over dk+el+f}
\Biggm|
{
a,b,c,d,e,f \in \mathbb{R},
\atop
dk+el+f > 0 \text{~for~} (k,l)\in \conv P
}
\right\}.
$$
We shall refer to this result as to {\em Graham algorithm.}
Unfortunately, for some objective functions~$\theta \in \Theta$
the algorithm fails and, as Graham writes, we should
``resort to manual calculations and ad hoc arguments''.
We discuss possible improvements in Section \ref{s:graham}.

The primary aim of the current paper is to provide an algorithm
to determine
\begin{equation}\label{eq:opt-problem-2}
\inf_{(k,l)\in P} \theta(k,l)
\end{equation}
under a nonempty set of linear constraints (thus $j\ne0$) and
\begin{equation}\label{eq:theta}
\theta = \max\{ \theta_1(k,l),\ldots,\theta_m(k,l) \},
\qquad \theta_1,\ldots,\theta_m \in \Theta.
\end{equation}

In Section \ref{s:projective} a useful computational
concept of projective exponent pairs is explained.
Section \ref{s:explore} is devoted to the exploration
of the geometry of $P$ and its results
are of separate interest.
In Section \ref{s:graham}
Graham algorithm is discussed.
Section \ref{s:algorithm} contains the description
of the proposed algorithm to solve \eqref{eq:opt-problem-2}
under linear constraints and \eqref{eq:theta}.
In Section \ref{s:consequences} new estimates and theoretical
results on various divisor problems are given, derived from
the observation of particular cases of the output
of our algorithm.

\section{Projective exponent pairs}\label{s:projective}

Let us map exponent pairs into the real projective space
(the concept of such mapping traces back to Graham
\cite{graham1986}):
$$
\mu\colon
\R^2 \to \R^3/(\R\setminus\{0\}),
\quad
(k,l) \mapsto (k:l:1).
$$
For the set of the exponent pairs
the inverse mapping
$$  \mu^{-1}\colon (k:l:m) \mapsto (k/m, l/m) $$
is also well-defined.

Operators $A$ and $BA$ are mapped by $\mu$
into linear operators over projective space:
$$
A(k,l) \mapsto \A (k:l:1),
\quad
\A = \begin{pmatrix} 1&0&0 \\ 1&1&1 \\ 2&0&2 \end{pmatrix},
\quad
\A(k:l:m) = \begin{pmatrix}k \\ k+l+m \\ 2k+2m \end{pmatrix}
$$
and
$$
BA(k,l) \mapsto \BA(k:l:1),
\quad \!\!\!
\BA =
\begin{pmatrix}
 0&1&0\cr
 2&0&1\cr
 2&0&2\cr
\end{pmatrix},
\quad
\BA(k:l:m) =
\begin{pmatrix}
l \\ 2k+m \\ 2k+2m
\end{pmatrix}.
$$
Thus $A=\mu^{-1} \A \mu$ and $BA = \mu^{-1} \BA \mu$

Such projective mappings are very useful to achieve better
computational performance.

Firstly,
we replace fractional calculations with integer ones.

Secondly, let $M$ be a fixed composition of $A$ and $BA$.
We can evaluate $Mp$ for a set of points $p$ effectively:
once precompute the matrix of the projective operator~$\M$
and then just calculate $\mu^{-1}\M\mu p$ for each point $p$.

\section{Exploring exponent pairs}\label{s:explore}

Let us split $Pp$ into {\em generations} $P_n p$ such that
$$
P_0 p = \{p\}, \qquad P_n p = AP_{n-1}p \cup BAP_{n-1}p,
\qquad n>0.
$$

Let us investigate properties of $P(0,1)$. As soon as
$$
A(0,1) = (0,1), \qquad BA(0,1) = (1/2,1/2),
$$
$$
A(1/2,1/2) =  BA(1/2,1/2) = (1/6,2/3)
$$
we obtain
$$
P(0,1) = \bigl\{(0,1), (1/2,1/2)\bigr\} \cup P(1/6,2/3).
$$
So it is enough to study $P(1/6,2/3)$.

All initial exponent pairs satisfy inequalities
$$ k+l\le1, \qquad k\le1/2, \qquad l\ge1/2. $$
One can check that if $(k,l)$ satisfies such inequalities,
then $A(k,l)$ and $BA(k,l)$ also do. Thus all exponent
pairs fits into the triangle
\begin{equation}\label{eq:triangle}
T:= \triangle\bigl( (1/2,1/2), (0,1), (0,1/2) \bigr).
\end{equation}

\begin{lemma}\label{l:lemma1}
Denote
$$
P' = (0,1/6)\times(2/3,1),
\qquad
P'' = (1/6,1/2)\times(1/2,2/3).
$$
Let $p:=(k,l)$ be the exponent pair such that $A p \in P'$. Then
\begin{equation}\label{eq:lemma1-statement}
A Pp \subset P',
\qquad
BA Pp \subset P'',
\qquad
Pp \subset \{p\} \cup P' \cup P''.
\end{equation}
\end{lemma}
\begin{proof}
Suppose that \eqref{eq:lemma1-statement} is true
for all generations $P_m$, $m< n$. Let us prove that
it is also true for generation $P_n$.

We have $B P' = P''$, so it is enough to prove
that $AP_{n-1}p \subset P'$. Let $(k,l)$ be an arbitrary
element of $P_{n-1}p$ and let $(\kappa, \lambda) = A(k,l)$.
There are three possibilities:
\begin{enumerate}
\item $(k,l) = p$. Then $(\kappa,\lambda)\in P'$
by conditions of the lemma.

\item $(k,l)\in P'$. Then
\begin{align*}
\kappa &= {1\over2} - {1\over2(k+1)}
< {1\over2} - {3\over7} = {1\over14}, \\
\lambda &= {1\over2} + {l\over2(k+1)}
> {1\over2} + {2/3\over7/3} = {11\over14}.
\end{align*}

\item $(k,l)\in P''$. Then
\begin{align*}
\kappa &= {1\over2} - {1\over2(k+1)}
< {1\over2} - {1\over3} = {1\over6}, \\
\lambda &= {1\over2} + {l\over2(k+1)}
> {1\over2} + {1/2\over3} = {2\over3}.
\end{align*}
\end{enumerate}
\end{proof}

An exponent pair $(1/6,2/3)$ satisfies conditions
of Lemma~\ref{l:lemma1}, because
$$
A(1/6, 2/3) = (1/14,11/14).
$$

We note that the statement of Lemma \ref{l:lemma1}
can be refined step-by-step, obtaining 4, 8, 16 and so on
rectangles, covering $Pp$ more and more precisely.

\begin{remark}\label{remark}
There exists another approach to cover $Pp$
or the whole $P$. For a set of
pairs $(\alpha_i, \beta_i)$ determine
with the use of Graham algorithm
$$ \theta_i = \inf_{(k,l)\in \conv Pp} (\alpha_i k + \beta_i l). $$
Then $Pp$ is embedded into a polygonal area, constrained
with the set of inequalities
$$ \alpha_i k + \beta_i l \ge \theta_i $$
from the bottom and left
(together they form a hyperbola-like line)
and by the segment from $(0,1)$ to $(1/2,1/2)$.
\end{remark}

\medskip

Let us introduce an order $\prec$ on $P(1/6,2/3)$, defined as
$$
(k,l) \prec (\kappa, \lambda)
\iff
k < \kappa, ~ l > \lambda.
$$

\begin{theorem}\label{th:strict-total-order}
Let $p$ be the exponent pair from the statement
of Lemma~\ref{l:lemma1}.
Then the order $\prec$ is a strict total order on $P_n p$
and this order coincides with the order
of the binary Gray codes \cite[Ch. 7.2.1.1]{knuth2011}
over an alphabet $\{A,BA\}$.
\end{theorem}
\begin{proof}
One can directly check that operator $A$ saves the order:
$$
p_1 \prec p_2 \Rightarrow A p_1 \prec A p_2
$$
and operator $B$ reverses it, so $BA$ reverses it too:
$$
p_1 \prec p_2 \Rightarrow BA p_1 \succ BA p_2.
$$
Lemma \ref{l:lemma1} implies that
for every $p_1,p_2\in P_{n-1}p$
\begin{equation}\label{eq:order-a-ba}
A p_1 \prec BA p_2.
\end{equation}
Combining these facts we obtain the statement of the theorem.
\end{proof}

In the case of $P(1/6,2/3)$ inequality \eqref{eq:order-a-ba}
can be refined up to
\begin{equation}\label{eq:order-a-1/6,2/3-ba}
A p_1 \prec (1/6,2/3) \prec BA p_2.
\end{equation}
Thus $\prec$ is a strict total order
over the whole $P(1/6,2/3)$.

\medskip

Fig. \ref{fig:p1623} illustrates our results.
Point $(1/6,2/3)$ divides the set into rectangles~$P'$
and $P''$. These rectangles consists of pairs, where the last applied
operator was $A$, and pairs, where the last applied operator
was $BA$, respectively. All plotted points are
total-ordered by $\prec$. Writing out points
of the same generation from the left top corner
to the right bottom corner we obtain a list of Gray codes.
E.~g., for the generation 3 we obtain
a sequence of 8 codes:

\nobreak

\centerline{
\begin{tabular}{r@{}r@{}r}
$A$&$A$&$A$,\\
$A$&$A$&$BA$,\\
$A$&$BA$&$BA$,\\
$A$&$BA$&$A$,\\
$BA$&$BA$&$A$,\\
$BA$&$BA$&$BA$,\\
$BA$&$A$&$BA$,\\
$BA$&$A$&$A$.\\
\end{tabular}
}

\noindent
\begin{figure}
\includegraphics[width=\textwidth]{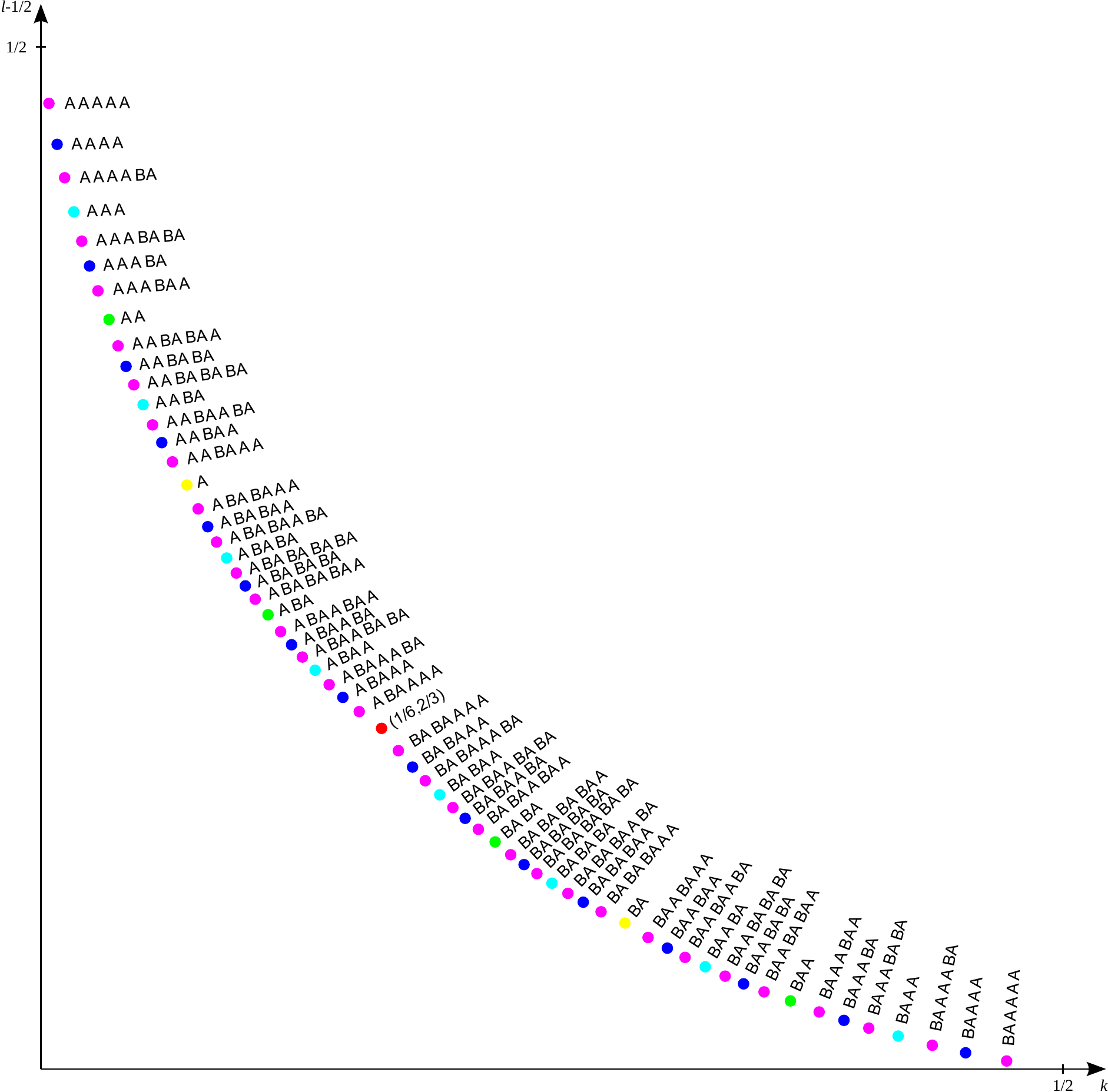}
\caption{First six generations of $P(1/6,2/3)$
plotted in shifted coordinates $(k,l-1/2)$.}
\label{fig:p1623}
\end{figure}

\medskip

As soon as
$$
A^n (1/6,2/3) \to (0+{},1-0) \qquad \text{as~} n\to\infty
$$
we obtain
\begin{align*}
p_{n} := A\cdot BA \cdot A^n(1/6,2/3) &\to (1/6-0,2/3+0), \\
BA\cdot BA \cdot A^n(1/6,2/3) &\to (1/6+0,2/3-0)
\qquad \text{as~} n\to\infty.
\end{align*}
So no point from $P(1/6,2/3)$ is isolated:
for every $p\in P(1/6,2/3)$ and every~$\eps>0$
there exist $p_1, p_2 \in P(1/6,2/3)$ such that
$p_1 \prec p \prec p_2$, $|p-p_1|<\eps$ and~$|p-p_2|<\eps$.

We are even able to compute the slopes of left-hand
and right-hand ``tangents'' at $(1/6,2/3)$. Namely,
using Section \ref{s:projective}
and denoting $d_n=p_n-(1/6,2/3)$ we get
$$
d_n/|d_n| \to (-2/\sqrt5,1/\sqrt5)
\qquad \text{as~} n\to\infty,
$$
so the left-hand ``tangent'' at $(1/6,2/3)$
has a slope $\arctan(-1/2)$. The right-hand ``tangent''
has a slope $\arctan(-2)$.

\medskip

What about sets generated from other known initial
exponent pairs, listed in Proposition \ref{pr:exp-pairs}?
Lemma \ref{l:lemma1} and Theorem \ref{th:strict-total-order}
remains valid. But inequality \eqref{eq:order-a-1/6,2/3-ba}
does not hold and so $\prec$ is not a strict total order.
E.~g., for
$$
p=A\cdot BA\cdot A^4 H_{05} = \left( {8083\over50342}, {1\over2}+{4304\over25171} \right)
$$
neither $p \prec H_{05}$, nor~$p \succ H_{05}$.

As opposed to $P(1/6,2/3)$, each point of the set $Pp$,
$p\ne(1/6,2/3)$, is isolated, because the initial point is.
But for every such $p$ each point of $P(1/6,2/3)$
has an arbitrary close to it point from $Pp$.

\medskip

\begin{lemma}\label{l:contraction}
Operators $A$ and $BA$ are contractions over the triangle $T$,
which was defined in \eqref{eq:triangle}.
\end{lemma}
\begin{proof}
It is enough to prove that $A$ is a contraction.
Let us check that there exists~$\alpha<1$
such that for each~$p_1,p_2\in T$ we have
$$
\bigl|A p_1-A p_2\bigr| \le \alpha \bigl|p_1-p_2\bigr|.
$$
Let $(k_1,l_1):=p_1$ and $(k_2,l_2):=p_2$. Then
\begin{multline*}
\bigl|A p_1-A p_2\bigr|^2
= {1\over4} \left( \left({1\over k_1+1}-{1\over k_2+1}\right)^2
+ \left({l_1\over k_1+1}-{l_2\over k_2+1}\right)^2  \right)
= \\
= {1\over4} \left( \left( k_2-k_1 \over (k_1+1)(k_2+1) \right)^2
+ \left( (l_1-l_2)(k_2+1)+l_2(k_2-k_1)
  \over (k_1+1)(k_2+1) \right)^2 \right).
\end{multline*}
But $k_1,k_2\ge0$, so
$$
\bigl|A p_1-A p_2\bigr|^2
\le {1\over4} \left( (k_1-k_2)^2
+ \left( |l_1-l_2| + |k_1-k_2| \right)^2 \right).
$$
Applying inequality $(x+y)^2 \le 2(x^2+y^2)$ we finally obtain
$$
\bigl|A p_1-A p_2\bigr|^2
\le {3\over4} \left( (k_1-k_2)^2 + (l_1-l_2)^2 \right)
= {3\over4} \bigl|p_1-p_2\bigr|^2.
$$
\end{proof}

\section{Notes on Graham algorithm}\label{s:graham}

Below GX means a reference to \cite[Step X at p.~209]{graham1986}.

\medskip

1. Graham algorithm is designed to search $\inf_{p\in P(0,1)} \theta p$ and relies on the fact that
$$
P(0,1) = A P(0,1) \cup BA P(0,1).
$$
This kind of decomposition does not hold for the whole $P$. Instead we have
$$
P = A P \cup BA P \cup \bigl( P_0 \setminus \{ (0,1) \} \bigr).
$$
Thus in order to run Graham algorithm over $P$, not just over $P(0,1)$,
it should be changed in following way. Establish a variable $r$
to keep a current minimal value, setting it initially to~$+\infty$.
Add an additional step before G5: apply current $\theta$
on elements of $P_0$ and set~$r \gets \min(r, \min \theta P_0)$.
At the end of the algorithm output $r$ instead of simply $\min \theta P_0$.

\medskip

2. Unfortunately, Graham algorithm over $P$ is infinite: no analog of halting conditions at G3 provided by \cite[Th. 3]{graham1986} can be easily derived. So we should stop depending on whether the desired accuracy is achieved. Cf. Step~\ref{item:abort} in the Section~\ref{s:algorithm} below.

\medskip

3. Bad news: if \cite[Th. 1, 2]{graham1986} does not specify the branch to choose at G4 then the original Graham algorithm halts. Good news: \cite[Th. 2]{graham1986} can be generalized to cover a wider range of cases. In notations of the mentioned theorem for a given finite sequence $M\in\{A,BA\}^n$ if $\inf \theta BA = \inf \theta BA M$ and if
$$
\min (r w+v-u, \alpha w+v-u) \ge 0
$$
then $\inf \theta = \inf \theta A$, where
$$
\alpha := \max\bigl\{ k+l \bigm| (k,l) \in AM\{(0,1),(1/2,1/2),(0,1/2)\} \bigr\}.
$$

\medskip

4. For the case of linearly constrained optimization one can build a ``greedy'' modification of Graham algorithm: if at G5 one of the branches is entirely out of constraints then choose another one; otherwise choose a branch in a normal way. Such algorithm executes pretty fast, but misses optimal pairs sometimes.

\section{Linear programming algorithm}\label{s:algorithm}

Now let us return to the optimization
problem \eqref{eq:opt-problem-2}. We will attack it
with the use of backtracking.

Operators $A$ and $BA$ perform projective mappings
of the plane~${\mathbb R}^2$, so both of them map straight
lines into lines and polygons into polygons.

Let $\theta$ be as in \eqref{eq:theta} and a set of linear
constraints $LC$ be as in \eqref{eq:opt-problem}.

Denote
\begin{equation*}
\theta_+(V)
= \max \Bigl\{ \sup_{p\in V} \theta_i p \Bigr\}_{i=1}^m,
\qquad
\theta_-(V)
= \max \Bigl\{ \inf_{p\in V} \theta_i p \Bigr\}_{i=1}^m.
\end{equation*}
Then
$$
\theta_-(V) \le \inf_{p\in V} \theta p \le  \theta_+(V)
$$
and these bounds embrace $\inf_{p\in V} \theta p$ tighter
and tighter as $V$ becomes smaller. Both~$\theta_+$
and $\theta_-$ can be computed effectively by simplex method.

Let $V$ be a polygon (or a set of polygons, lines
and points) such that~$P \subset V$.
See Lemma \ref{l:lemma1} and paragraphs above and below it
for possible constructions of~$V$.
For a set of linear constraints $LC$ let $R(V,LC)$
be a predicate, which is true if and only if there exists
a point $p\in V$, which satisfies all constraints from~$LC$.
This predicate
can be computed effectively using algorithms
for line segment intersections~\cite[p.~19--44]{berg2008}.

The proposed algorithm consists of a routine
$L(\theta, LC, r, \M)$, which calls itself recursively.
Here $r$ keeps a current minimal value of $\theta(k,l)$
and $\M$ is a current projective transformation matrix.
Initially $r \gets +\infty$ and
$\M \gets \left(
\begin{smallmatrix} 1&0&0\\0&1&0\\0&0&1 \end{smallmatrix}
\right)$.

\medskip

On each call routine $L$ performs following steps:
\begin{enumerate}
\item
Compute $ t \gets \min \{ \theta \mu^{-1} \M \mu p \}$,
where $p$ runs over all known initial pairs~$p$.
If we get $t<r$ then update current minimal value: $r\gets t$.
\item \label{item:abort}
Check whether the desired accuracy is achieved,
comparing $r$ with the values of $\theta_+(\mu^{-1}\M\mu V)$
and $\theta_-(\mu^{-1}\M\mu V)$.
If yes then return $r$ and abort computations.
\item \label{item:lc'}
Set
$LC' \gets LC \cup \{ \theta_i(k,l) < r \}_{i=1}^m$.
Due to the nature of $\theta_i \in \Theta$
a constraint of form~$\theta_i(k,l) < r$ is in fact a linear constraint.
\item
If $R(\mu^{-1}\M\A\mu V, LC')$ (that means that there is
at least a chance to meet exponent pair
$p\in \mu^{-1}\M\A\mu V$, which satisfies $LC$
and on which objective function is less
than yet achieved value) then compute
$t\gets L(\theta, LC, r, \M \mathcal{A})$.
If $t<r$ set~$r\gets t$ and recompute $LC'$ as in
Step \ref{item:lc'} using the new value of $r$.
\item
If $R(\mu^{-1}\M\BA\mu V, LC')$ then compute
$t\gets L(\theta, LC, r, \M \mathcal{BA})$.
If $t<r$ set~$r\gets t$.
\item
Return $r$.
\end{enumerate}

The algorithm executes in finite time, because due to
Lemma~\ref{l:contraction} both $A$ and $BA$ are contractions
and sooner or later (depending on required accuracy)
recursively called routines will abort at Step \ref{item:abort}.

Step \ref{item:lc'} plays a crucial role in chopping off
non-optimal branches of the exhaustive search and preventing exponential running time. We are not able to provide any theoretical estimates, but in all our experiments (see Section~\ref{s:consequences} below) the number of calls of $L(\cdot, \cdot, \cdot, \cdot)$ behaved like a linear function of the recursion's depth.

We have implemented our algorithm as a program, written
in PARI/GP~\cite{parigp}. It appears that it runs pretty fast,
in a fraction of a second on the modern hardware.

During computations elements of $\M$ can grow enormously.
As soon as matrix~$\M$ is applied on projective vectors
we can divide $\M$ on the greatest common divisor
of its elements to decrease their magnitude.

\medskip

Now, under which circumstances an equality
$$
\inf_{p \in P} \theta p
=
\inf_{p \in \conv P} \theta p
$$
holds? Certainly it is true for $\theta\in\Theta$  and $j=0$,
because sets $\{\theta p=\const\}$ are straight lines;
this is the case of Graham algorithm.

Consider the case $\theta \in \Theta$ and $j\ne0$. Constraining lines are specified by equations
$$
l_i = \{ \alpha_i k + \beta_i l + \gamma_i = 0 \}, \qquad i=1,\ldots,j.
$$
Then
$$
\inf_{p\in\conv P} \theta p = \min \left\{
\inf_{p\in P} \theta p,
\inf_{p\in l_1 \cap \conv P} \theta p,
\ldots,
\inf_{p\in l_j \cap \conv P} \theta p
\right\}.
$$
But $\conv P$ is approximated by a polygon as in Remark \ref{remark}, so $l_i \cap \conv P$ can be approximated too
and consists of a single segment. Thus
$\inf_{p\in l_i \cap \conv P} \theta p$ is computable.

The case when $\theta$ is as in \eqref{eq:theta}
with $m>1$ is different. Even without any constraints
the value of $\inf_{p\in\conv P} \theta p$ may be not equal
to $\inf_{p\in P} \theta p$. For example, take
$$
\theta(k,l) = \max \bigl\{ 11k/10, l-1/2 \bigr\}
.$$
Then
$$
\inf_{p\in P} \theta p = {176 \over 1025}
\qquad \text{at~}
p=H_{05}.
$$
But
$$
\inf_{p\in\conv P} \theta p = {176\over 1057 }
\qquad \text{at~}
p=(160/1057,1409/2114):=q,
$$
and $q$ is owned by
a segment from $(0,1)$ to~$H_{05}$. However, in not-so-synthetic
cases the proposed algorithm produces results, which are closer
to optimal.

\section{Applications}\label{s:consequences}

One can run algorithm from the previous section to obtain
numerical results in partial cases for different objective
functions and constraints. It gives us a way to catch site
of some patterns and to suppose general statements on them.
Nevertheless these patterns should be proved,
not only observed. This is the main theme
of the current section.

Consider the asymmetrical divisor problem.
Denote
$$
\tau(a_1,\ldots,a_k; n)
= \sum_{d_1^{a_1}\cdots d_k^{a_k} = n} 1,
$$
which is called {\em an asymmetrical divisor function.}
Let $\Delta(a_1,\ldots,a_k; x)$ be an error term
in the asymptotic estimate of the
sum~$\sum_{n\le x} \tau(a_1,\ldots,a_k; n)$.
(See~\cite{kratzel1988} for the form of the main term.)
What upper estimates of $\Delta$ can be given?
The following result is one of the possible answers.

\begin{theorem}[{\cite[Th. 5.11]{kratzel1988}}]\label{th:tauab}
Let $a<b$ and let $(k,l)=A(\kappa, \lambda)$ be an exponent pair. Then the estimate
\begin{equation*}
\Delta(a,b;x) \ll x^{\alpha} \log x,
\qquad \alpha = {2(k+l-1/2) \over (a+b)}
\end{equation*}
holds under the condition $(2l-1)a \ge 2 k b$.
Here $f(x) \ll g(x)$ denotes $f(x) = O\bigl(g(x)\bigr)$.
If otherwise $(2l-1)a < 2 k b$, then
\begin{equation*}
\Delta(a,b;x) \ll x^{\alpha} \log x,
\qquad \alpha = {k\over(1-l)a+k b}.
\end{equation*}
\end{theorem}

Taking into account Lemma \ref{l:lemma1} the condition
$(k,l)=A(\kappa,\lambda)$ can be rewritten
as~$k<1/6$ and $l>2/3$. Thus
\begin{align*}
\theta_1 = {2(k+l-1/2)\over a+b}, &&
LC_1 = \bigl\{ (2l-1)a \ge 2 k b, k<1/6, l>2/3 \bigr\},
\\
\theta_2 = {k\over(1-l)a+k b}, &&
LC_2 = \bigl\{ (2l-1)a < 2 k b, k<1/6, l>2/3 \bigr\}.
\end{align*}
Using proposed algorithm we can compute $\inf\theta_1$
under constraints $LC_1$ (which refers to the first case
of Theorem \ref{th:tauab}), compute $\inf\theta_2$
under constraints $LC_2$ (which refers to the second case)
and take lesser of the obtained values.
Observed results shows that for $a=1$, $b=2^r$, $r\ge10$,
the second case provides better results and exponent pair
has form
$$ A^{r-1} BA A^{r-4} BA BA \ldots. $$
This leads us to the following statement.

\begin{theorem}\label{th:sum-of-tau1m-precise}
For a fixed integer $r\ge5$
we have $\Delta(1,2^r;x) \ll x^{\alpha} \log x$, where
$$
\alpha = {2^r-2r \over 2^{2r} - r\cdot2^r - 2r^2 + 2r - 4}
< {1\over 2^r+r}.
$$
\end{theorem}
\begin{proof}
Consider an exponent pair
$$
(k_r, l_r) := A^{r-1} BA A^{r-4} (1/6, 2/3).
$$
We have
$$
\A = {\mathcal S} \begin{pmatrix}
	1&1&0 \\ 0&1&0 \\ 0&0&2
\end{pmatrix} {\mathcal S}^{-1},
\qquad
{\mathcal S} = \begin{pmatrix}
	0&-1&0 \\ 1&0&1 \\ 0&2&1
\end{pmatrix}.
$$
Thus
$
\A^n = {\mathcal S} \left( \begin{smallmatrix}
	1&n&0 \\ 0&1&0 \\ 0&0&2^n
\end{smallmatrix} \right) {\mathcal S}^{-1}.
$
Note that $\mu(1/6,2/3) = (1:4:6)$ and
$$
\A^{r-1} \BA \A^{r-4} (1:4:6)
=
\begin{pmatrix}
2^r - 2r \\
2^{2r+1} - (3r+4) \cdot 2^r + 2r^2+2r+4 \\
2^{2r+1} - (2r+4) \cdot 2^r + 4r
\end{pmatrix}.
$$
Applying $\mu^{-1}$ we get
$$
k_r = {2^r-2r \over 2^{2r+1}-(2r+4)\cdot 2^r + 4 r},
\quad
l_r = 1 - { r\cdot 2^r - 2r^2+2r-4
	\over 2^{2r+1}-(2r+4)\cdot 2^r + 4 r}.
$$
Now for $r\ge5$
$$
2l_r - 2\cdot2^r k_r - 1 = {2r^2+4-2^{r+1}
	\over 2^{2r}-(r+2)\cdot 2^r + 2r} < 0.
$$
This proves that $(k_r, l_r)$ satisfies the second case
of Theorem \ref{th:tauab} and finally
$$
\alpha = {k_r \over 2^r k_r - l_r + 1}.
$$
\end{proof}

In the same manner one can estimate $\Delta(a,2^r;x)$ for odd $a$.
Here is one more example.

\begin{theorem}\label{th:sum-of-tau3m-precise}
For a fixed integer $r\ge 1$ we have
$\Delta(3,2^r;x) \ll x^{\alpha+\eps}$, where
$$
\alpha = {1\over 2^r+3r-88/17}.
$$
\end{theorem}
\begin{proof}
Consider an exponent pair
$
(k, l) := A^{r-3} BA A (9/56+\eps,37/56+\eps)
$.
\end{proof}

In the case of $\Delta(a,b,c)$ one can derive objective
function and constraints from \cite[Th. 6.2, 6.3]{kratzel1988}
and observe the output of the algorithm.

\begin{theorem}
For a fixed integer $r\ge10$ we have
\begin{equation*}\label{eq:sum-of-tau1mm-precise}
\theta(1,2^r,2^r) = {
26\cdot 2^{2r} - (29r+41) 2^r + 16r^2+12r+32
\over
26\cdot 2^{3r} - (16r+41) 2^{2r} + (24r-3) 2^r + 16r+12
} < {1\over 2^r+1}.
\end{equation*}
\end{theorem}
\begin{proof}
Follows from \cite[Th. 6.2]{kratzel1988} with
$
(k,l) = A^{r-1} B A^{r-2} BA BA^2 \cdot B (0,1)
$.
\end{proof}

\begin{table}
\begin{tabular}{cccc}
$a$ & $b$ & $(k,l)$ & $\Xi(a,b)$ \\
1 & 2 & $BA H_{05}$ & $269/1217$ \\
1 & 3 & $(BA)^2 A BA H_{05}$ & $1486/8647$ \\
1 & 4 & $H_{05}$         & $111/790$ \\
1 & 5 & $A BA A^2 BA A (BA)^2 A^2 M^\infty (0,1)$ & ${(15921-2 c) / 30437}$ \\
1 & 6 & $(A BA)^3 (BA)^3 A^3 BA (0,1)$ & $669/6305$ \\
1 & 7 & $A (BA)^2 BA A (BA)^2A^2 M^\infty (0,1) $ & ${(9370-c)/34469}$ \\
1 & 8 & $A (BA)^4 (A^2 BA A)^\infty (0,1)$ & ${(5+\sqrt{809}) / 392}$ \\
1 & 9 & $ A (BA)^2 A M^\infty (0,1) $ & ${(10551-c)/56976}$ \\
1 &10 & $A (BA)^2 (A^2 (BA)^2)^2 A BA H_{05} $ & $ 150509/2096993$ \\
2 & 3 & $BA A (BA)^2A^2 M^\infty (0,1)$ & ${(c-4047)/15688}$ \\
2 & 4 & $BA H_{05}$ & $269/2434$ \\
2 & 5 & $M^\infty  (0,1)$ & ${(c-4311)/18672}$ \\
3 & 4 & $BA A H_{05}$ & $1819/19369$ \\
3 & 5 & $BA A (BA)^3 A^2 (BA)^3 A (BA)^5 A^2 BA (0,1)$ & $63916/774807$ \\
4 & 5 & $BA A H_{05}$ & $1819/24903$ \\
\end{tabular}
\caption{Estimates of $\Xi$. Here
$M=(BA)^6 (A BA)^2 BA A^2$ and $c=\sqrt{37368753}$.}
\label{t:menzer-nowak}
\end{table}

Further, consider the asymmetric divisor problem
with congruence conditions on divisors. Namely,
let $\tau(a,m_a,r_a;b,m_b,r_b; n)$ be the number
of $(d_a,d_b)$ such that
$$
d_a^a d_b^b = n,
\qquad
d_a \equiv r_a \pmod {m_a},
\qquad
d_b \equiv r_b \pmod {m_b}.
$$
Menzer and Nowak showed in \cite{menzer1989} that
if $a < b$ then the error term in the asymptotic estimate of
$$
\sum_{n\le x} \tau(a,m_a,r_a;b,m_b,r_b; n)
$$
has form
$
\left( x / m_a^a m_b^b \right)^{\Xi(a,b)+\varepsilon}
$, where
$$
\Xi(a,b) := \inf_{(k,l)\in \mathop{\mathrm{conv}} P} \max
\left\{ {k+l\over (k+1)(a+b)}, {k\over k b+a(1+k-l) } \right\},
$$
where $\eps>0$ is arbitrary small. They also listed estimates
of $\Xi(a,b)$ for $1\le a<b\le5$. As soon as $\Xi(a,b)$ is
of form \eqref{eq:theta} we can refine all their results.
See Table \ref{t:menzer-nowak}.

\bigskip

Various estimates of the Riemann zeta function depends on optimization tasks \eqref{eq:opt-problem}. The following theorem seems to be the simplest example.

\begin{theorem}[{\cite[(7.57)]{ivic2003}}]
Let $\zeta$ denote the Riemann zeta function and $\sigma \ge 1/2$. Further, let $\mu(\sigma)$ be an infimum of all $x$ such that $\zeta(\sigma+it)  \ll t^x$. Then
$$
\mu(\sigma) \le {k+l-\sigma \over 2}.
$$
for every exponent pair $(k,l)$ such that $l-k \ge \sigma$.
\end{theorem}

Better results on $\mu$ leads to better estimates for power moments of $\zeta$, and the last are helpful to improve estimates in multidimensional divisor problem. See \cite[Th. 8.4, 13.2, 13.4]{ivic2003}.

Table \ref{table:zeta} contains several results on $\mu(\sigma)$ obtained with the use of the proposed algorithm. Results are accompanied with the number of calls to $L(\cdot, \cdot, \cdot, \cdot)$ up to the given depth of search.

\begin{table}
\centering
\begin{tabular}{ccccc}

$\sigma$ & $\mu(\sigma)$ & Depth 100 & Depth 1000 \\
3/5 & 1409/12170 & 10 & 10 \\
2/3 & 0.0879154 & 154 & 1609 \\
3/4 & 0.0581840 & 154 & 1610 \\
4/5 & 3/71 & 103 & 1003 \\

\end{tabular}
\caption{Estimates for $\mu(\sigma)$ and the number of calls to $L$.}
\label{table:zeta}
\end{table}

\bibliographystyle{ugost2008s}
\bibliography{taue}

\end{document}